\documentclass[12pt]{amsart}

\usepackage{latexsym, amssymb, amscd, amsfonts,pstricks,enumitem,amsmath}
\usepackage[all]{xypic}
\usepackage{ifthen}
\usepackage{longtable}

\renewcommand{\subsection}[1]{\vspace{3mm}\refstepcounter{subsection}\noindent{\bf \thesubsection. #1.} }

\renewcommand{\subsubsection}[1]{\vspace{3mm}\refstepcounter{subsubsection}\noindent{\bf \thesubsubsection. #1.} }

\numberwithin{equation}{section}

\newcommand{\Hom}{\operatorname{Hom}}
\renewcommand{\geq}{\geqslant}
\renewcommand{\leq}{\leqslant}
\newcommand{\Osh}{{\mathcal O}}                        %  Structure sheaf
\newcommand{\image}{\operatorname{image}}
\renewcommand{\H}{\mathrm{H}}                          %  Cohomology group
\newcommand{\V}{\mathrm{V}}                            
\newcommand{\id}{\operatorname{id}}                    %  the identity map
\newcommand{\Pic}{\operatorname{Pic}} %Pic
\newcommand{\lcm}{\operatorname{lcm}}
\newcommand{\supp}{\operatorname{supp}}

\newcommand{\idlim}{\varprojlim} %indirect (projective) limit
                        
\newcommand{\K}{\mathrm{K}}                            
\newcommand{\Span}{\operatorname{span}}                % the span of some vectors
\newcommand{\Aut}{\operatorname{Aut}} % automorphism 

\newcommand{\x}{\mathbf{x}}  %we need bold font x
\newcommand{\y}{\mathbf{y}} % we need bold font y
\newcommand{\z}{\mathbf{z}}  % we need bold font z
\newcommand{\e}{\mathbf{e}} % we need bold font e

\newcommand{\VV}{\operatorname{V}} %inverse system multiplication by n
\newcommand{\TT}{\operatorname{T}} % the tate module
\newcommand{\tor}{\operatorname{tor}} % torsion points functor

\newcommand{\Res}{\operatorname{Res} } % Restricted representation
\newcommand{\Ind}{\operatorname{Ind} } % Induced representation

\newcommand{\ii}{\operatorname{i}}
\newcommand{\G}{\mathrm{G}}
\newcommand{\NS}{\operatorname{NS}} %Neron-Severi

\newcommand{\End}{\operatorname{End}} % endomorphism

\renewcommand{\emptyset}{\varnothing}

\newcommand{\ZZ}{\mathbb{Z}} % integers

\newtheorem{theorem}{Theorem}[section]
\newtheorem{lemma}[theorem]{Lemma}
\newtheorem{corollary}[theorem]{Corollary}
\newtheorem{proposition}[theorem]{Proposition}
\theoremstyle{definition}

\begin{document}
\title{Refinements to Mumford's theta and adelic theta groups}

\author{Nathan Grieve}
\address{ Department of Mathematics and Statistics,
McGill University,
Montreal, QC, Canada}
\email{ngrieve@math.mcgill.ca}%{nathan.grieve@mail.mcgill.ca}

\thanks{As published in Annales math\'{e}matiques du Qu\'{e}bec, Vol. 38, No. 2, 145--167 (2014). }

\maketitle
\begin{abstract}
Let $X$ be an abelian variety defined over an algebraically closed field $k$.
We consider theta groups associated to \emph{simple semi-homogenous vector bundles of separable type} on $X$.  We determine the structure and representation theory of these groups.  In doing so we relate work of Mumford, Mukai, and Umemura.  We also consider adelic theta groups associated to line bundles on $X$.  After reviewing Mumford's construction of these groups we determine  functorial properties which they enjoy and then realize the Neron-Severi group of $X$ as a subgroup of the cohomology group $\H^2(\V(X), k^\times)$. 
\end{abstract}

\section{Introduction}

Let $X$ be an abelian variety defined over an algebraically closed field $k$.  Our purpose here is to refine and generalize the theory of theta-Heisenberg and adelic theta groups associated to line bundles on $X$.  Such groups were invented by Mumford and played a fundamental role in his study of syzygies of abelian varieties, moduli of abelian varieties, and theta functions, see \cite{MumI}, \cite{MumII}, \cite{MumIII}, and \cite{Mum:Quad:Eqns}.  

These groups continue to play an important role in the theory of abelian varieties especially in the study of syzygies, moduli, and vector bundles.  Some more recent examples include the work of Gross-Popescu, \cite[Example 2.10, p. 349]{G-P}, where theta groups and their higher weight representation theory play a role in their study of syzygies, work of  Nakamura, \cite{Nakamura:99}, where theta group schemes are used to compactify the moduli scheme of abelian schemes over $\operatorname{Spec} \ZZ[\zeta_N,1/N]$, and work of  Oprea, \cite[\S 2]{Oprea:2011}, where theta groups and their relation to semi-homogeneous vector bundles are used to study Verlinde bundles.  Even more recently Brion has considered theta groups associated to Brauer-Severi varieties over abelian varieties, \cite{Bri}, while Shin has extended Mumford's work by constructing theta and adelic theta groups associated to line bundles on abelian schemes, \cite{Shin}.

Here we make three contributions. Our first two concern theta groups associated to a class of simple vector bundles on $X$ which we refer to as \emph{simple semi-homogeneous vector bundles of separable type} and define in \S \ref{semi-homog:prelim}. Our first result, Theorem \ref{non-degen-vb}, concerns the structure of these groups. 
This theorem  generalizes \cite[\S 1, Theorem 1]{MumI} and answers a problem of Umemura, see \cite[p. 120]{Umemura:1973},  for this class of vector bundles on $X$. Our second result, 
Theorem \ref{thm:theta:vb:1} and Corollary \ref{unique:theta:irred}, determines the representation theory of these groups.   This theorem generalizes \cite[\S 1, Proposition 3 and Theorem 2]{MumI} and also makes the inequality asserted in \cite[Exercise 6.10.4, p. 175]{BL} an equality.  

The third result of this paper, Theorem \ref{adelic:NS}, concerns adelic theta groups associated to line bundles on $X$.
In \S \ref{adelic:theta} we review the construction of these groups as it is important to our proof of Theorem \ref{adelic:NS}.

To describe Theorem \ref{adelic:NS} let $I$ denote the set of positive integers which are not divisible by the characteristic of $k$.  Let
$\tor(X) := \{ x \in X(k) : n x = 0 \text{ for some $n \in I$} \} \text{,}$  and let $\V(X) := \idlim \tor(X)$, where the limit is indexed by $I$ and where the maps are given by multiplication by $n / m$ whenever $m$ divides $n$.  

In this notation Theorem \ref{adelic:NS} gives a functorial realization of the  Neron-Severi group of $X$ as a subgroup of the cohomology group $\H^2(\V(X), k^\times)$.

To place our results, concerning theta groups, in proper context it is important to emphasize that our results build on work of Mukai, especially \cite{Muk78}, and Mumford, namely \cite[\S 1]{MumI}.
For instance the concept of semi-homogeneity is due to Mukai, \cite[p. 239]{Muk78}.  In that paper he also characterized simple semi-homogeneous vectors bundles.  Indeed he proved that a simple vector bundle on $X$ is semi-homogeneous if and only if it is the direct image of some line bundle with respect to some isogeny, \cite[Theorem 5.8, p. 260]{Muk78}.   

Note also that prior to Mukai's work \cite{Muk78}, which builds on work of Oda \cite[Corollary 1.7, p. 53]{OdaElliptic71} and \cite[Lemma 1.4, p. 71]{Oda:deRham}, Umemura extended Mumford's theory of theta groups.  He considered theta groups associated to vector bundles on $X$ and determined the weight $1$ representation theory of theta groups associated to simple vector bundles on $X$.  In addition he posed the problem of determining the structure of theta groups associated to vector bundles on $X$ in general and considered theta groups associated to Brauer-Severi varieties over $X$, see \cite[\S 5]{Umemura:1973}, \cite[\S 1 and \S 2]{Umemura:1976}.

Here we say that a simple semi-homogeneous vector bundle is of \emph{separable type} if its Euler characteristic is nonzero and not divisible by the characteristic of the ground field, see \S \ref{semi-homog:prelim}.  In this context our Theorems \ref{non-degen-vb} and \ref{thm:theta:vb:1}  generalize Mumford's \cite[Theorem 1, p. 293]{MumI} and \cite[Theorem 2, p. 297]{MumI} which apply to ample line bundles of separable type on $X$.

As one final comment we mention that there is some overlap amongst our Theorem \ref{mainTheorem1}, which we use to prove Theorem \ref{thm:theta:vb:1}, and work of Goren \cite[Appendix]{Goren:theta:preprint}.  Also a special case of Theorem \ref{non-degen-vb} is implicit in \cite[\S 2.2]{Oprea:2011}.  On the other hand all of the results of this paper were obtained independently in my dissertation \cite{Grieve:PhD:Thesis} and I am not aware of any other reference which states and proves these results explicitly.

\noindent{\bf Acknowledgements.}  
I thank my Ph.D. adviser Mike Roth for useful discussions.  The final writing of this work benefited from conversations with Eyal Goren, Jacques Hurtubise, and criticisms and suggestions given by the referee.

\section{Statement of results}\label{statement:results}
In \S \ref{semi-homog:prelim} we describe the class of vector bundles for which our results concerning theta groups apply.  These results are stated in \S \ref{main:results}.    In \S \ref{main:results2} we formulate our results concerning adelic theta groups.

\subsection{Preliminaries concerning simple semi-homogeneous vector bundles}\label{semi-homog:prelim}
Let $X$ be an abelian variety, defined over an algebraically closed field $k$, and let $T_x : X \rightarrow X$ denote translation by $x \in X$.  
The first two results of this paper, which we state in \S \ref{main:results}, concern theta groups associated to \emph{simple semi-homogeneous vector bundles of separable type} on $X$.

We make this concept precise as follows.  To begin with let $E$ be a vector bundle on $X$.   Then, following Mukai \cite[p. 239]{Muk78}, we say that $E$ is \emph{semi-homogeneous} if for all $x \in X$ there exists a line bundle $L$ on $X$ with the property that $T^*_x E \cong E \otimes L$.  

In the case that $E$ is \emph{simple}, that is if $\dim_k \H^0(X,E^\vee \otimes E) = 1$,  Mukai proved that $E$ is semi-homogeneous if and only if there exists an isogeny $\pi : Y \rightarrow X$ and a line bundle $L$ on $Y$ with the property that $E \cong \pi_* L$, \cite[Theorem 5.8, p. 260]{Muk78}.

In \cite[\S 2.1, p. 6]{Grieve-cup-prod-ab-var} we defined $E$ to be \emph{non-degenerate} if its Euler  characteristic $\chi(E)$ is nonzero.  We then proved, \cite[Proposition 2.1, p. 6]{Grieve-cup-prod-ab-var}, that if $E$ is simple semi-homogenous and non-degenerate then $E$ admits exactly one nonzero cohomology group $\H^{\ii(E)}(X,E)$.  Thus the cohomology groups of non-degenerate simple semi-homogenous vector bundles behave in a manner analogous to those of non-degenerate line bundles.

Here, motivated by the above considerations, as well as Mumford's concept of ample line bundles of separable type \cite[p. 289]{MumI}, we make one more definition which is relevant to what we do here.

\noindent
{\bf Definition.}  Let $E$ be a vector bundle on $X$.  We say that
$E$ is of \emph{separable type} if its Euler characteristic $\chi(E)$ is nonzero and not divisible by the characteristic of $k$.

\subsection{Results concerning theta groups}\label{main:results} Let $E$ be a simple semi-homogeneous vector bundle of separable type on $X$.    If $x \in X$, then we let $\operatorname{Aut}_x(E)$ denote the set of isomorphisms $E \rightarrow T^*_x E$ of $\Osh_X$-modules.  

 In \cite[\S 6, \S 7, and Corollary 7.9, p. 271]{Muk78} Mukai has shown that the group
$$\K(E) := \{ x \in X(k) : \operatorname{Aut}_x(E) \not = \emptyset\}$$
is finite and that $\chi(E)^2 = \# \K(E)$.

The theta group of $E$ is defined as
$ \G(E) := \{(x,\phi) : x \in \K(E) \text{ and } \phi \in \Aut_x(E) \}$ with multiplication given by 
 $(x,\phi) \cdot (y, \psi) := (x + y, \phi * \psi)$ where $\phi * \psi$ is the element of $\Aut_{x + y}(E)$ determined by the composition
 $$ E \xrightarrow{\psi} T^*_y E \xrightarrow{T^*_y \phi} T^*_y(T^*_x E) = T^*_{x + y} (E) \text{.}$$
The group $\G(E)$ is a central extension\footnote{Throughout we employ the (somewhat non-standard) terminology of \cite{MacLane:homology}: if $A$ and $C$ are groups and $1\rightarrow A \rightarrow B \rightarrow C \rightarrow 1$ is a short exact sequence of groups then we say that $B$ is an extension of $A$ by $C$.} of $k^\times$ by $\K(E)$ and acts with weight one on $\H^{\operatorname{i}(E)}(X,E)$.  

Our first result, Theorem \ref{non-degen-vb}, generalizes \cite[Theorem 1]{MumI} and sheds light on a problem posed by Umemura \cite[p. 120]{Umemura:1973}.  It concerns the structure of the group $\G(E)$ and is stated as:

\begin{theorem}\label{non-degen-vb}
Let $E$ be a simple semi-homogenous vector bundle of separable type on $X$.  The theta group $\G(E)$ is a non-degenerate central extension of $k^\times$ by $\K(E)$.
\end{theorem}

A consequence of Theorem \ref{non-degen-vb} is that every simple semi-homogenous vector bundle of separable type $E$ on $X$ determines a sequence of integers $d = (d_1,\dots, d_p)$, $d_{i+1} \mid d_{i}$, which we refer to as the \emph{type of its theta group $\G(E)$}.  This fact, together with Theorem \ref{non-degen-vb}, allows us to determine the representation theory of $\G(E)$.

%:thm:theta:vb:1
\begin{theorem}\label{thm:theta:vb:1}
Let $E$ be a simple semi-homogenous vector bundle of separable type on $X$.  Let $d = (d_1,\dots, d_p)$ be the type of $\G(E)$.  The following assertions hold:
\begin{enumerate}
\item{the theta group $\G(E)$ admits exactly $\gcd(n,d_1)^2 \times \cdots \times \gcd(n,d_p)^2$ non-isomorphic irreducible weight $n$ $\G(E)$-module(s);}
\item{a weight $n$ representation is irreducible if and only if it has dimension $\frac{d_1 \times \cdots  \times d_p}{\gcd(n,d_1) \times \cdots \times \gcd(n,d_p)}$;}
\item{every weight $n$ $\G(E)$-module decomposes into a direct sum of irreducible weight $n$ $\G(E)$-modules.  Every $\G(E)$-module decomposes into a direct sum of weight $n$ $\G(E)$-modules.}
\end{enumerate}
\end{theorem}

As explained in \S \ref{semi-homog:prelim}, each simple semi-homogeneous vector bundle of separable type on $X$ admits exactly one nonzero cohomology group. 
Combining this fact with Theorems \ref{non-degen-vb} and \ref{thm:theta:vb:1} we obtain:

\begin{corollary}\label{unique:theta:irred}
Let $E$ be a simple semi-homogenous vector bundle of separable type on $X$.  The unique nonzero cohomology group $\H^{\ii(E)}(X, E)$ is the unique irreducible weight $1$ representation of its theta group $\G(E)$.
\end{corollary}

We prove Theorem \ref{non-degen-vb} in \S \ref{theta:loc:free}, while we prove Theorem \ref{thm:theta:vb:1} and Corollary  \ref{unique:theta:irred} in \S \ref{proof:theta:rep:vb}.

\subsection{Results concerning adelic theta groups}\label{main:results2}
Let $X$ be an abelian variety defined over an algebraically closed field $k$.
Let $I$ denote the set of positive integers which are not divisible by the characteristic of $k$.  Let 
$$\tor(X) := \{ x \in X(k) : n x = 0 \text{ for some $n \in I$} \} \text{,}$$  and let $\V(X) := \idlim \tor(X)$, where the limit is indexed by $I$ and where the maps are given by 
$$[ n / m ] : \tor(X) \rightarrow \tor(X)$$ whenever $m$ divides $n$.  

Let $L$ be the total space of a line bundle on $X$.  Mumford has constructed a group $\widehat{\G}(L)$, the \emph{adelic theta group of $L$}, see \cite[\S 7]{MumII} and \cite[Chapter 4]{theta3}; this group is a central extension of $k^\times$ by $\V(X)$. 
We recall some aspects of Mumford's construction of $\widehat{\G}(L)$ in \S \ref{adelic:theta}.  

In \S \ref{proof:adelic:theta} we prove that the Neron-Severi group of $X$, which we denote by $\NS(X)$, can be canonically identified with a subgroup of the cohomology group $\H^2(\V(X), k^\times)$.  Here $\H^2(\V(X), k^\times)$ denotes the usual second cohomology  group of the trivial $\V(X)$-module $k^\times$.  
More specifically we establish:  

\begin{theorem}\label{adelic:NS}
The map $\widehat{\G} : \NS(X) \rightarrow \H^2(\V(X),k^\times)$, defined by sending the class of a line bundle $L$ to that of its adelic theta group $\widehat{\G}(L)$, is an injective group homomorphism.  It satisfies the following functorial property: if $f : X \rightarrow Y$ is a homomorphism of abelian varieties then the diagram
$$ \xymatrix{ \NS(X) \ar[r]^-{\widehat{\G}} & \H^2(\V(X),k^\times) \\
\ar[u]^-{f^*} \NS(Y) \ar[r]^-{\widehat{\G}} & \H^2(\V(Y),k^\times) \ar[u]^-{f^*}}$$ commutes.
\end{theorem}

\section{Theta groups and quasi-coherent sheaves}

Let $X$ be an abelian variety defined over an algebraically closed field $k$.  In this section we construct theta groups, associated to quasi-coherent sheaves on $X$, and determine some of their basic properties.

\subsection{Preliminaries from descent theory}\label{descent}  
Let $K \subseteq X$ be a finite subgroup and assume that the order of $K$ is not divisible by the characteristic of $k$.  We consider descent of quasi-coherent sheaves with respect to the quotient map $f : X \rightarrow Y = X / K$.

Let $\mathrm{Descent}(\mathrm{QCoh}(X),f)$ denote the category whose objects are descent data $(F,\phi)$, with respect to $f$, and where a morphism between pairs $(F,\phi)$ and $(G,\psi)$ is given by an element $\alpha \in \Hom_{\Osh_X}(F,G)$ which has the property that the diagram
$$ \xymatrix{ F \ar[d]_-{\phi_x} \ar[r]^-{\alpha} &  G \ar[d]^-{\psi_x} \\
			T^*_x F \ar[r]^-{T^*_x \alpha} & T^*_x G  } $$
commutes for all $x \in K$.		

Grothendieck has proven that the pullback functor $f^* : \mathrm{QCoh}(Y) \rightarrow \mathrm{Descent}(\mathrm{QCoh}(X),f)$ is an equivalence of categories.   A proof of this result can be found in \cite[\S 6.1 Theorem 4, p. 134]{B-L-R}, for instance, or \cite[\S 7 Proposition 2, p. 66 and \S 12 Theorem 1, p. 104]{Mum}.

In \S \ref{theta:loc:free} we apply the following elementary observation in our proof of Lemma \ref{simple:homog:lemma}.
 \begin{lemma}\label{decent:lemma}
 If a coherent sheaf $E$ on $X$ descends, via a separable isogeny $f : X \rightarrow Y$ to a coherent sheaf $F$ on $Y$, then 
 $ \dim_k \End_{\Osh_X}(E) \geq \dim_k \End_{\Osh_Y}(F)\text{.}$
 \end{lemma}
 \begin{proof}
Since the descent functor is fully faithful the $k$-vector space $\End_{\Osh_Y}(F)$ is identified with the subspace of $\End_{\Osh_X}(E)$ which commutes with the descent data.
 \end{proof}

%: theta:quasi:coherent
\subsection{Construction of theta groups}\label{theta:quasi:coherent}
Let $F$ be a quasi-coherent sheaf on $X$.  If $x \in X(k)$ then let 
$\Aut_x(F) := \{ \phi \in \Hom_{\Osh_X}(F, T^*_x F) : \text{$\phi$ is an isomorphism}\}\text{.}$
If $\Aut_x(F) \not = \emptyset$ then it is an $\Aut(F)$-torsor (or principal homogeneous space).  
Let 
$$\K(F) := \{ x \in X(k) : \Aut_x(F) \not = \emptyset \} $$
and observe that $\K(F)$ is a subgroup of $X(k)$.
 The theta group of $F$ is
$$ \G(F) := \{(x,\phi) : x \in \K(F) \text{ and } \phi \in \Aut_x(F) \}$$ where multiplication is defined by 
 $(x,\phi) \cdot (y, \psi) := (x + y, \phi * \psi)$, where $\phi * \psi$ is the element of $\Aut_{x + y}(F)$ determined by the composition
 $$ F \xrightarrow{\psi} T^*_y F \xrightarrow{T^*_y \phi} T^*_y(T^*_x F) = T^*_{x + y} (F) \text{.}$$
The homomorphisms 
 $\text{$\iota_F : \Aut(F) \rightarrow \G(F)$,  and $\pi_F : \G(F) \rightarrow \K(F)$,}$ defined respectively by $\alpha \mapsto (0, \alpha)$ and  $(x, \phi) \mapsto x$, determine a short exact sequence of groups
 $$ 1 \rightarrow \Aut(F) \xrightarrow{\iota_F } \G(F) \xrightarrow{\pi_F} \K(F) \rightarrow 0 $$
 and makes $\G(F)$ an extension of $\Aut(F)$ by $\K(F)$.
 
\subsection{Level subgroups and descent} 
Let $F$ be a quasi-coherent sheaf on $X$.  Here we generalize \cite[Proposition 1, p. 291]{MumI} and relate certain subgroups of $\G(F)$ and descent data for $F$ with respect to suitably defined isogenies.  

\noindent
{\bf Definitions} (\cite[p. 291]{MumI}){\bf.} \begin{itemize}
\item{A subgroup $\mathcal{K} \subseteq \G(F)$ is a \emph{level subgroup} if it is finite, has order not divisible by the characteristic of $k$, and if the homomorphism $\pi_F : \mathcal{K} \rightarrow \pi_F(\mathcal{K})$ is injective.}
\item{Let $K \subseteq \K(F)$ be a subgroup and let $\mathcal{K} \subseteq \G(F)$ be a level subgroup.  We say that $\mathcal{K}$ \emph{lies over} $K$ if $\pi_F(\mathcal{K}) = K$.}
\end{itemize}

\begin{proposition}[Compare with {\cite[Proposition 1, p. 291]{MumI}}]  Let $F$ be a quasi-coherent sheaf on $X$ and let $K \subseteq \K(F)$ be a subgroup.
There exists a one to one correspondence between level subgroups $\mathcal{K} \subseteq \G(F)$ lying over $K$ and (effective) descent datum $(F,\phi)$ with respect to the quotient map $f : X \rightarrow X /  K$.
\end{proposition}
\begin{proof}
Let $\sigma$ be the inverse of $\pi_F $.  If $x \in K$, then let $\phi_x$ be the element of $\Aut_x(F)$ determined by $\sigma$.  Let $\phi$ be the set consisting of these isomorphisms.  Since $\sigma$ is a group homomorphism the diagram
\begin{equation}\label{descent:commute}\xymatrix{ F \ar[r]^-{\phi_x}  \ar[drr]_-{\phi_{x + y}} & T^*_x F \ar[r]^-{T^*_x \phi_y} & T^*_x (T^* _y F)   \\ 
 &  &  T^*_{x + y} F \ar@{=}[u] }
\end{equation}
%\eqref{descent:commute} 
 commutes for all $x,y \in K$.  Hence the pair $(F,\phi)$ is descent datum.  

Conversely if $(F,\phi)$ is descent datum then define 
$\mathcal{K} := \{(x,\phi_x) : \phi_x \in \phi \}\text{.} $  Since $(F,\phi)$ is descent datum the diagram \eqref{descent:commute} commutes, for all $x,y \in K$, and we deduce that this implies that $\mathcal{K}$ is a subgroup.  In addition the map $(x, \phi_x) \mapsto x$ is an isomorphisms of groups.

Finally it is clear, by construction, that the correspondences just defined are mutual inverses.
\end{proof}

\subsection{Theta groups and isogenies}
Let $F$ be a quasi-coherent sheaf on $X$, $\mathcal{K} \subseteq \G(F)$ a level subgroup, and  $(F,\phi)$ the descent datum determined by $\mathcal{K}$.  Let $K := \pi_F(\mathcal{K})$, $Y := X / K$, and $f : X \rightarrow Y$ the quotient map.  

Since $(F,\phi)$ is effective there exists a quasi-coherent sheaf $H$ on $Y$ with the property that $(F,\phi)$ is isomorphic to $(f^*H,\operatorname{can}(H))$, the canonical descent data determined by $f^*H$.  As a consequence there exists an isomorphism $\alpha : f^* H \rightarrow F$ of $\Osh_X$-modules.

We now use $\alpha$ to relate $\G(H)$, $\G(F)$, and the centralizer $\operatorname{C}_{\mathcal{K}}(\G(F))$ of $\mathcal{K}$ in $\G(F)$.  This is the content of Proposition \ref{descent:prop}.  

Before stating this result 
first observe that every $x \in X$ determines a morphism
\begin{equation}\label{pull:back:aut:1}
\Aut_{f(x)}(H) \xrightarrow{f^*} \Aut_x (f^* H) \xrightarrow{T^*_x(\alpha) \circ ? \circ \alpha^{-1}} \Aut_x(F) 
\end{equation}
of $\Aut(H)$-sets.  In particular, we have $f^{-1}(\K(H)) \subseteq \K(F)$. 

Also if $x$ and $y$ are elements of $X$ then the diagram
\begin{scriptsize}
\begin{equation}\label{pull:back:aut:2}
\xymatrixcolsep{7pc}\xymatrix{
\Aut_{f(y)}(H) \times \Aut_{f(x)}(H) \ar[r]^-{f^* \times f^*} \ar[d]^-{*} & \Aut_y(f^* H) \times \Aut_x(f^* H) \ar[r]^-{T^*_y(\alpha) \circ ? \circ \alpha^{-1} \times T^*_x(\alpha) \circ ? \circ \alpha^{-1}} \ar[d]^-{*} & \Aut_y(F) \times \Aut_x(F) \ar[d]^-{*} \\
\Aut_{f(x) + f(y)}(H) \ar[r]^-{f^*} & \Aut_{x + y}(f^* H) \ar[r]^-{T^*_{x + y}(\alpha) \circ ? \circ \alpha^{-1}} & \Aut_{x + y}(F) 
}
\end{equation}
\end{scriptsize}
of $\Aut(H)$-sets commutes.

Proposition \ref{descent:prop}, which generalizes \cite[Proposition 2, p. 291]{MumI}, can now be stated.  

\begin{proposition}[Compare with {\cite[Proposition 2, p. 291]{MumI}}]
\label{descent:prop}  
The following assertions hold:
\begin{enumerate}
\item{if $\operatorname{C}_{\mathcal{K}}(\G(F))$ denotes the centralizer of $\mathcal{K}$ in $\G(F)$ then $\operatorname{C}_{\mathcal{K}}(\G(F))$ coincides with the set
$$ \{ (x,\eta) : f(x) \in \K(H) \text{ and } \eta = T^*_x(\alpha) \circ f^* \psi \circ \alpha^{-1} \text{ for some $\psi \in \Aut_{f(x)}(H)$} \}\text{;}$$ }
\item{the map $\operatorname{C}_{\mathcal{K}}(\G(F)) \rightarrow \G(H)$ defined by $(x, \eta) \mapsto (f(x),\psi)$, where $\psi$ is the (unique) element of $\Aut_{f(x)}(H)$ whose pullback is $\eta$, is a surjective group homomorphism with kernel equal to $\mathcal{K}$.
}
\end{enumerate}
\end{proposition}

\begin{proof}
We first prove (a). 
If $w \in K$ then let $\phi_w \in \Aut_w(F)$ be the unique isomorphism $F \rightarrow T^*_w F$ having the property that $(w,\phi_w) \in \mathcal{K}$.

We know that $(F, \phi)$ descends to $H$ and that if $x \in X$ and $y = f(x)$ then $(T^*_x F, T^*_x \phi)$ descends to $T^*_y H$. 

The centralizer of $\mathcal{K}$ in $\G(F)$ consists exactly of those $(x,\psi) \in \G(F)$ with the property that $\psi *  \phi_w = \phi_w * \psi$ for all $w \in K$.  
Considering this condition we conclude that $\operatorname{C}_{\mathcal{K}}(\G(F))$ consists exactly of those $(x,\psi) \in \G(F)$ such that $\psi$ determines an isomorphism amongst the pairs $(F,\phi) $ and $(T^*_x F,T^*_x \phi)$.

Thus, to determine $\operatorname{C}_{\mathcal{K}}(\G(F))$, we need to examine, for a fixed $x \in \K(F)$, those isomorphisms $\psi : F \rightarrow T^*_x F$ which commute with the descent data.

Let $x \in X$ and let $y = f(x)$.  
The map
\begin{equation}\label{pull:back:aut:3} \Hom_{\Osh_Y}(H, T^*_y H) \rightarrow \Hom((F,\phi),(T^*_x F, T^*_x \phi))
\end{equation}
is given by 
$\eta \mapsto T^*_x(\alpha) \circ f^* \eta \circ \alpha^{-1}$ and is an isomorphism (since the descent functor is fully faithful).  

Furthermore under the map \eqref{pull:back:aut:3} isomorphisms carry over to isomorphisms.  In particular if $\psi : F \rightarrow  T^*_x F$ is an isomorphism which commutes with descent data then $f(x)$ is an element of $ \K(H)$.    Considering the discussion above we conclude that (a) holds.

To prove (b), using the diagram \eqref{pull:back:aut:2}, we check that the asserted map is a group homomorphism.  To see that the asserted map is surjective let $y \in \K(H)$.   Then $y = f(x)$ for some $x \in \K(F)$.  Since the map \eqref{pull:back:aut:3} is an isomorphism, every element of $\Aut_y(H)$ is in the image.  
Using the definition of the map we check that its kernel is $\mathcal{K}$.
\end{proof}

\noindent
{\bf Remark.}  Using the fact that $\mathcal{K}$ is a level subgroup of $\G(F)$ we can check that the centralizer of $\mathcal{K}$ in $\G(F)$ equals its normalizer.

\section{Non-degenerate theta groups}

Let $X$ be an abelian variety defined over an algebraically closed field $k$.  
Let $E$ be a simple semi-homogeneous vector bundle of separable type on $X$.  In \S \ref{theta:loc:free} we prove that its theta group $\G(E)$ is a non-degenerate central extension of $k^\times$ by $\K(E)$.

\subsection{Preliminaries on central extensions of $k^\times$ by a finite abelian group}\label{non-deg:central:extensions}  
Let $\K$ be a finite abelian group and assume that the order of $\K$ is not divisible by the characteristic of $k$.  Here we recall some facts about central extensions
\begin{equation}\label{basic:central}
1 \rightarrow k^\times \xrightarrow{\iota} \G \xrightarrow{\pi} \K \rightarrow 0
\end{equation}
of $k^\times$ by $\K$.

To begin with the extension \eqref{basic:central} is said to be \emph{non-degenerate} if $\iota(k^\times)$ equals the center of $\G$.  In addition we say that a subgroup $\mathcal{K} \subseteq \G$ is a \emph{level subgroup} if the homomorphism $\pi : \mathcal{K} \rightarrow \pi(\mathcal{K})$ is injective.
%: commutator-central-extension

Important to the structure of the extension \eqref{basic:central} is the commutator of $\G$ which determines a bilinear form 
\begin{equation}\label{commutator:bilinear}
 [-,-]_{\G} \colon \K \times \K \rightarrow k^\times \text{, defined by 
$[x,y]_{\G} := \tilde{x} \tilde{y} \tilde{x}^{-1} \tilde{y}^{-1}$,}\end{equation} 
where $\tilde{x}$ and $\tilde{y}$ are any elements of $\G$ lying over $x$ and $y$ respectively.
 
It is clear that the form \eqref{commutator:bilinear} is skew symmetric.    
Also, as explained in \cite[p. 293]{MumI}, the bilinear form $[-,-]_{\G}$ is related to level subgroups of $\G$.   More specifically if $K \subseteq \K$ is a subgroup having the property that $[x,y]_{\G} = 1$ for all $x,y \in K$, then $\G$ admits a level subgroup $\mathcal{K}$  with the property that $\pi(\mathcal{K}) = K$.
Finally if the extension \eqref{basic:central} is non-degenerate, then the form \eqref{commutator:bilinear} is also non-degenerate; this fact has the following important consequence.

%: normal-forms

\begin{proposition}[{\cite[p. 293-294]{MumI}}]\label{normal:form:theta}
If $\G$ is a non-degenerate central extension of $k^\times$ by $\K$ then $\K$ admits subgroups $K_1,K_2 \subseteq \K$ with the properties that
\begin{enumerate}
\item{$\K = K_1\oplus K_2$ and $[x_i,y_i]_{\G} =1$ if $x_i,y_i \in K_i$, for $i=1,2$;}
\item{the bilinear form 
$$\langle -,-\rangle \colon K_1\times K_2\rightarrow k^\times \text{, defined by $(x_1,x_2)\mapsto [x_1, x_2]_{\G}$}$$ is non-degenerate;} 
\item{ the extension $\G$ is equivalent to the extension $\G_{\langle -,- \rangle}$, where
$\G_{\langle -,-\rangle} := k^\times \times K_1\oplus K_2$ and where multiplication is defined by 
$$(\alpha,x_1,x_2)\cdot (\beta, y_1,y_2)=(\alpha \beta \langle x_1,y_2 \rangle, x_1+y_1,x_2+y_2)\text{.}$$ 
}
\end{enumerate} 
\end{proposition}
\begin{proof}
Use the arguments described in \cite[p. 293-294]{MumI}.
\end{proof}

One consequence of Proposition \ref{normal:form:theta} is that we can associate a sequence of integers to  every non-degenerate theta group.  Indeed, suppose that $\G$ is a non-degenerate central extension of $k^\times$ by $\K$.  If $\G$ is abelian then $\K$ is trivial and we say that $\G$ is of \emph{type} $(1)$.  Otherwise, by Proposition \ref{normal:form:theta}, the invariant factors of $\K$ occur in pairs. In this case, let 
$(d_1,\dots, d_p)$ be a sequence of positive integers with the properties that:
\begin{enumerate}
\item{ $d_{i + 1} | d_i$, $d_i > 1$;}
\item{$(d_1,d_1, \dots, d_p,d_p)$ are the invariant factors of $\K$.}
\end{enumerate}  We then say that $\G$ is of \emph{type} $(d_1,\dots, d_p)$.

%: theta:loc:free
\subsection{Non-degenerate theta groups and vector bundles}\label{theta:loc:free}
Let $X$ be an abelian variety defined over an algebraically closed field $k$ and let $E$ be a simple vector bundle on $X$.  

We have homomorphisms 
$\text{
$\iota_E : k^\times \rightarrow \G(E)$ and $\pi_E : \G(E) \rightarrow \K(E)$}$ defined, respectively, by $\alpha \mapsto (0,\alpha \id_E)$ and $(x,\phi) \mapsto x$.  These homomorphisms determine a short exact sequence of groups 
$$1 \rightarrow k^\times \xrightarrow{\iota_E}  \G(E) \xrightarrow{\pi_E} \K(E) \rightarrow 0 \text{.}$$
Since the image of $\iota_E$ is contained in the center of $\G(E)$, the group $\G(E)$ is a central extension of $k^\times$ by $\K(E)$.

Now suppose that $E$ is a simple semi-homogeneous vector bundle of separable type on $X$.   Then $\K(E)$ is a finite group and $\chi(E)^2 = \#\K(E)$, \cite[\S 6, 7, and Corollary 7.9, p. 271]{Muk78}.  

Our goal here is to prove Theorem \ref{non-degen-vb} which implies that $\G(E)$ is a non-degenerate theta group.  Before doing so let us establish one additional lemma.

\begin{lemma}\label{simple:homog:lemma}
Let $f : X \rightarrow Y$ be a separable isogeny and let $E$ be a simple semi-homogeneous vector bundle of separable type on $X$.
If $E$ descends, via $f$, to a vector bundle $F$ on $Y$, then $F$ is a simple semi-homogeneous vector bundle of separable type on $Y$.  
\end{lemma}
\begin{proof}

Mukai's theory implies that $F$ is semi-homogeneous, see \cite[Proposition 5.4, p. 259]{Muk78}.  To prove that $F$ is simple, using Lemma \ref{decent:lemma}, we obtain the inequalities
$$ 1 = \dim_k \End_{\Osh_X}(E) \geq \dim_k \End_{\Osh_Y}(F) \geq 1$$ 
and conclude that $F$ is simple.  

On the other hand, we have that
$ \chi(E) = (\# \ker f) \chi(F)$
while 
$\chi(E)^2 = \# \K(E)$ and  $\chi(F)^2 = \# \K(F)\text{.}$  
We conclude that $\chi(F) \not = 0$ and that $\chi(F)$ is not divisible by the characteristic of $k$.
\end{proof}

We now use Lemma \ref{simple:homog:lemma} and Proposition \ref{descent:prop} to prove Theorem \ref{non-degen-vb}.  The proof of this theorem is similar to Mumford's proof of  \cite[Theorem 1, p. 293]{MumI}.

\begin{proof}[Proof of Theorem \ref{non-degen-vb}]
Let $\mathcal{K}$ be a maximal level subgroup of $\G(E)$ and 
let $K := \pi_E(\mathcal{K})$.  Then $K$ is a maximal subgroup of $\K(E)$ on which $[-,-]_{\G(E)} \equiv 1$.  

Let $Y := X / K$, let $f : X \rightarrow Y$ denote the quotient map, and let $F$ be a vector bundle on $Y$ to which $E$ descends.  (Such an $F$ is determined by $\mathcal{K}$.)  Then, by Lemma \ref{simple:homog:lemma}, $F$ is a simple semi-homogeneous vector bundle of separable type on $X$.  

Using the fact that $\mathcal{K}$ is maximal, together with Proposition \ref{descent:prop}, we check that 
$$\G(F) = (k^\times \cdot \mathcal{K} ) / \mathcal{K}$$ and conclude that $\K(F)$ is trivial.  

Since $\chi(F)^2 = \# \K(F)$ we conclude that $\chi( F )^2 = 1$. Since $\chi( E )^2 = (\# K)^2\chi (F)^2$ we conclude that $\chi(E)^2 = (\# K)^2$.  Since $\mathcal{K}$ was an arbitrary maximal level subgroup, we conclude that $\chi(E)^2 = (\# \mathcal{K})^2$ for every maximal level subgroup $\mathcal{K}$ of $\G(E)$.

Now let $K_0 := \{ x \in \K(E) : [x,y]_{\G(E)} = 1 \text{ for all $y \in \K(E)$}\}$.  Then $[-,-]_{\G(E)}$ induces a non-degenerate skew-symmetric bilinear form
$$ \K(E) / K_0 \times \K(E) / K_0 \rightarrow k^\times \text{.}$$  Hence $\# (\K(E) / K_0) = \ell^2$, for some $\ell$, and
%, by Proposition \ref{normal:bilinear}, 
there exists a maximal subgroup $K'$ of $\K(E) / K_0$, of order $\ell$, on which $[-,-]_{\G(E)} \equiv 1$.

Let $K$ be the inverse image of $K'$ in $\K(E)$.  Then $K$ is the image of a maximal level subgroup of $\G(E)$.  We conclude that 
$$\chi(E)^2 = (\# K)^2 = (\# K_0)^2 \cdot \ell^2$$ and hence that
$$ \chi( E) ^2  = (\# K_0)^2 \cdot \ell^2 \\
 = (\# K_0)^2 \cdot \#(\K(E) / K_0) \\
 = (\# K_0)\#(\K(E)) \\
 = (\# K_0) \cdot \chi(E)^2 \text{.}
$$
Hence $\# K_0 = 1$ which implies that $K_0$ is trivial.
\end{proof}

\section{The representation theory of non-degenerate theta groups}\label{theta:reps}

Let $\K$ be a finite abelian group and assume that the characteristic of $k$ does not divide the order of $\K$.  Let $\G$ be a non-degenerate central extension of $k^\times$ by $\K$.  In this section we determine the representation theory of $\G$.

Before proceeding we fix some terminology.  A $\G$-module $(V,\rho)$ is always
a finite dimensional $k$-vector space which admits  a basis $\mathcal{B}$ for which there exists Laurent polynomials $F_{i,j} \in k [t,t^{-1}]$ with the property that the matrix representation of $\rho(\alpha)$, $\alpha \in k^\times$, with respect to $\mathcal{B}$ is given by evaluating $F_{i,j}$ at $\alpha$. We say that a $\G$-module $(V, \rho)$ is a \emph{weight $n$ $\G$-module}, for $n \in \ZZ$, if $\rho(\alpha) \cdot v = \alpha^n v$ for all $v \in V$ and all $\alpha \in k^\times$.

In \S \ref{proof:mainTheorem1} we prove the following theorem which we use to establish Theorem \ref{thm:theta:vb:1} and Corollary \ref{unique:theta:irred}, see \S \ref{proof:theta:rep:vb} for more details.

\begin{theorem}\label{mainTheorem1}  Let $\K$ be finite abelian group and assume that the characteristic of $k$ does not divide the order of $\K$.  Let $\G$ be a non-degenerate central extension of $k^\times$ by $\K$ of type  $(d_1,\dots, d_p)$.  
There exists exactly $\gcd(n,d_1)^2 \times \cdots \times \gcd(n,d_p)^2$ non-isomorphic irreducible weight $n$ $\G$-module(s).
A weight $n$ representation is irreducible if and only if it has dimension $\frac{d_1 \times \cdots \times d_p}{\gcd(n,d_1) \times \cdots \times \gcd(n,d_p)}$.
Every weight $n$ $\G$-module decomposes into a direct sum of irreducible weight $n$ $\G$-modules.  Every $\G$-module decomposes into a direct sum of weight $n$ $\G$-modules.
\end{theorem}

To determine the representation theory of $\G$, considering Proposition \ref{normal:form:theta}, it suffices to determine the representation theory of the group $\G_{\langle - , - \rangle}$.  In what follows we omit the subscript $\langle - , - \rangle$ and denote $\G_{\langle - , - \rangle}$ simply by $\G$. 

Let $(d_1,\dots, d_p)$ be the type of $\G$ and define
$D_n := \frac{d_1 \times \cdots \times d_p}{\gcd(n,d_1) \times \cdots \times \gcd(n,d_p)}$, for each $n \in \ZZ$.

\subsection{Auxiliary results and proof of Theorem \ref{mainTheorem1}}\label{proof:mainTheorem1}  Our proof of Theorem \ref{mainTheorem1} is similar to Mumford's proof of \cite[Proposition 3, p. 295]{MumI}, which determines the weight $1$-representation theory of $\G$.  
Central to our proof of Theorem \ref{mainTheorem1} is:

\begin{proposition}\label{key:Proposition}
If $(V, \rho)$ is a nonzero weight $n$ $\G$-module, then $(V,\rho)$ admits a $D_n$-dimensional submodule.
\end{proposition}

\begin{proof}
Let $(V,\rho)$ be a weight $n$ $\G$-module.  An element $x$ of $K_1$ acts on a vector $v \in V$ by the rule $x \cdot v := \rho((1,x,0))(v)$.  We denote the resulting $K_1$-module by $\Res^{\G}_{K_1}(V)$.  

Since $K_2 = \Hom_\ZZ(K_1,k^\times)$ the $K_1$-module $\Res^{\G}_{K_1}(V)$ admits a decomposition into eigenspaces.  Explicitly, we have 
\begin{equation}\label{AA}
\Res^{\G}_{K_1}(V) = \bigoplus_{y \in K_2} V_y
\end{equation}
 and if $x \in K_1$, $y \in K_2$, and $v \in V_y$, then $x \cdot v = \langle x ,y \rangle v$.

Observe now that, if $y \in K_2$, $v \in V_y$, and $(\alpha,x,w) \in \G$, then 
\begin{equation}\label{key:equation}
\rho((\alpha,x,w))(v) \in V_{y+nw}\text{.}
\end{equation}

Let $\pi_n \colon K_2 \rightarrow K_2$ denote the group homomorphism defined by $y \mapsto ny$.  The image of $\pi_n$ has order $D_n$. 

Using \eqref{key:equation}, we see that every  $V_y$, $y \in K_2$, appearing in the decomposition \eqref{AA}, is stable under the (evident) action of $\ker \pi_n$.  As a consequence if $y \in K_2$ then the $\ker \pi_n$-module $V_y$ decomposes into eigenspaces

\begin{equation}\label{BB}
V_y = \bigoplus_{\chi \in \Hom_{\ZZ}(\ker \pi_n, k^\times)} V_{y,\chi} \text{.}
\end{equation}

Let $\sigma$ be a set-theoretic section of the surjective homomorphism $K_2 \rightarrow \image \pi_n$ induced by $\pi_n$.
If $z \in \image \pi_n$, $y \in K_2$, $\chi \in \Hom_\ZZ(\ker \pi_n, k^\times)$, and $s_{y,\chi} \in V_{y,\chi}$, then define 
$$ s_{y,\chi}(z) := \rho((1,0,\sigma(z)))(s_{y,\chi})\text{.}$$
Observe now that 
\begin{equation}\label{CCC}
\rho((\alpha,x,w))(s_{y,\chi}(z)) = \alpha^n \langle x, y+z \rangle \chi(w + \sigma(z) - \sigma(nw + z)) s_{y,\chi}(nw + z)
\end{equation}
for all $(\alpha,x,w) \in \G$.

Since $(V,\rho)$ is nonzero, there exists $y \in K_2$ and $\chi \in \Hom_\ZZ(\ker \pi_n, k^\times)$ such that $V_{y, \chi} \not = 0$.  Fix such a pair $(y,\chi)$, choose a nonzero vector $s_{y,\chi} \in V_{y,\chi}$, and define 
$$W^{\sigma}_{y,\chi} := \operatorname{span}_{k}\{s_{y,\chi}(z) \}_{z \in \image \pi_n} \text{.}$$  Using equation \eqref{CCC}, we see that $W^\sigma_{y,\chi}$ is a $D_n$-dimensional $\G$-submodule of $V$.
\end{proof}

%: consequences:key:prop
\begin{corollary}\label{cor:key:prop}
\noindent
\begin{enumerate}
\item{A weight $n$ $\G$-module is irreducible if and only if it has dimension $D_n$.}
\item{Every weight $n$ $\G$-module decomposes into irreducible weight $n$ $\G$-modules.}
\end{enumerate}
\end{corollary}
\begin{proof}
To prove (a) note that if $V$ is an irreducible weight $n$ $\G$-module then it is nonzero and hence, by Proposition \ref{key:Proposition}, admits a submodule of dimension $D_n$.  Since $V$ is irreducible this submodule must equal $V$ whence the dimension of $V$ equals $D_n$.

Conversely let $V$ be a weight $n$ $\G$-module of dimension $D_n$.  Let $W$ be a nonzero submodule.  By Proposition \ref{key:Proposition}, $W$ admits a submodule of dimension $D_n$.  Consequently we have that
$$ D_n \leq \dim W \leq \dim V = D_n\text{.}$$  Hence $W$ has dimension $D_n$ so $W = V$.  We conclude that $V$ is irreducible.

To prove (b) let $\mu_{d_1}\subseteq k^\times$ denote the multiplicative group of $d_1$th roots of unity.   Let $G'$ denote the subgroup
$$G':=\{(\alpha,x,y) : \alpha \in \mu_{d_1}, x\in K_1, y \in K_2\}$$ of $\G$.   

To finish the proof of Corollary \ref{cor:key:prop}, we induct on the dimension of $V$. The base case is $\dim V=0$ in which case the assertion holds.  If $\dim V = N$, then combining Proposition \ref{key:Proposition} and part (a), which we just proved, we see that $V$ admits an irreducible weight $n$ submodule $W$.  If $W = V$ then we are done.  Otherwise choose a projection $p_0: V \rightarrow W$ and let $p:V\rightarrow W$ be the projection defined by $$v\mapsto \frac{1}{|G'|}\sum_{g\in G'} g \cdot p_0(g^{-1}\cdot v) \text{.}$$  Then $\ker p$ is a $G'$-submodule of $V$ and $V=W\oplus \ker p$.  Let $s\in \ker p$.  Then, if $(\alpha,x,y)\in \G$, we obtain
$$\rho((\alpha,x,y))(s)=\rho((\alpha,0,0))(\rho((1,x,y))(s))=\alpha^n(\rho((1,x,y))(s)).$$  Since $(1,x,y)\in G'$ and since $\ker p$ is $G'$-stable we conclude that
$$\alpha^n(\rho((1,x,y))(s))\in \ker p \text{.}$$  Hence $\ker p$ is a weight $n$ $\G$-submodule of $V$ and,  by induction, $\ker p$ decomposes into  irreducible weight $n$ $\G$-modules.
\end{proof}

%:existence:irred
To construct irreducible weight $n$ $\G$-modules first let 
$y \in K_2 \text{, } \chi \in \Hom_{\ZZ}( \ker \pi_n, k^\times) \text{,}$ and 
$ W_{y,\chi} := \Span_{k} \{ \mathbf{e}_{y+z,\chi} \}_{z \in \image \pi_n}\text{.}$

As in the proof of Proposition \ref{key:Proposition}, we fix a set-theoretic section $\sigma$ of the surjective homomorphism $K_2 \rightarrow \image \pi_n$ induced by $\pi_n$.  For every $(\alpha,x,w) \in \G$ define an automorphism
$$\rho^{\sigma,n}_{y,\chi}((\alpha,x,w)) : W_{y,\chi} \rightarrow W_{y,\chi} $$
by
$\e_{y+z,\chi} \mapsto \alpha^n \langle x, y+z \rangle \chi(w + \sigma(z) - \sigma(nw +z)) \e_{y+z+nw,\chi}$
and extending linearly.

\begin{proposition}\label{existence:Prop}
The pair $(W_{y,\chi},\rho^{\sigma,n}_{y,\chi})$ is an irreducible weight $n$ $\G$-module.
\end{proposition}
\begin{proof}
It is clear from the definition of $\rho^{\sigma,n}_{y,\chi}$ that $(W_{y,\chi},\rho^{\sigma,n}_{y,\chi})$ is a weight $n$ $\G$-module.
Since $k^\times$ acts with weight $n$ and since $W_{y,\chi}$ has dimension $D_n$, Corollary \ref{cor:key:prop} implies that the pair $(W_{y,\chi},\rho^{\sigma,n}_{y,\chi})$ is an irreducible weight $n$ $\G$-module. 
\end{proof}

%:characterize:irred
We now characterize irreducible weight $n$ representations.

\begin{proposition}\label{number:non:equiv}
A weight $n$ $\G$-module $(V,\rho)$ is irreducible if and only if it is isomorphic to $(W_{y,\chi},\rho^{\sigma,n}_{y,\chi})$ for some $y \in K_2$ and some $\chi \in \Hom_{\ZZ}(\ker \pi_n, k^\times)$.  Furthermore $(W_{y,\chi},\rho^{\sigma,n}_{y,\chi})$ is isomorphic to $(W_{y',\chi'},\rho^{\sigma,n}_{y',\chi'})$ if and only if $y - y' \in \image \pi_n$ and $\chi = \chi'$.
\end{proposition}
\begin{proof}
If $(V,\rho)$ is irreducible then it equals the subspace 
$W^\sigma_{y,\chi} := \Span_k \{ s_{y,\chi}(z)\}_{z \in \image \pi_n}$ constructed in the proof of Proposition \ref{key:Proposition}, for some $y \in K_2$, for some element $\chi$ of  $\Hom_{\ZZ}(\ker \pi_n,k^\times)$, and for some nonzero vector $s_{y,\chi} \in V_{y,\chi}$.  

Identifying the basis vectors $\{ s_{y,\chi}(z)\}_{z \in \image \pi_n}$ of $W^\sigma_{y,\chi}$ with those  
$\{\e_{y+z,\chi}\}_{z \in \image \pi_n} $ of $W_{y,\chi}$, and computing the matrix representations of $\rho$ and $\rho^{\sigma,n}_{y,\chi}$ with respect to these bases, we conclude that $(V,\rho)$ is isomorphic to $(W_{y,\chi},\rho^{\sigma,n}_{y,\chi})$.

If $y - y' \in \image \pi_n$ and $\chi = \chi'$ then we conclude that $(W_{y',\chi'},\rho^{\sigma,n}_{y',\chi'})$ is isomorphic to $(W_{y,\chi},\rho^{\sigma,n}_{y,\chi})$ by considering their matrix representations with respect to the bases 
$$ \{ \e_{y' + z,\chi} \}_{z \in \image \pi_n} \text{ and } \{ \e_{y + z,\chi} \}_{z \in \image \pi_n} $$ reordering one of them if necessary.  

Conversely if $(W_{y,\chi},\rho^{\sigma,n}_{y,\chi})$ is isomorphic to $(W_{y',\chi'},\rho^{\sigma,n}_{y',\chi'})$ then they are isomorphic as $K_1$-modules and as $\ker \pi_n$-modules.  If they are isomorphic as $K_1$-modules, then every $y+z$, $z \in \ker \pi_n$ equals $y' + z'$ for some $z' \in \ker \pi_n$.  In particular, $y - y' \in \ker \pi_n$. If they are isomorphic as $\ker \pi_n$-modules, then $\chi = \chi'$. 
\end{proof}

\begin{proof}[Proof of Theorem \ref{mainTheorem1}]
The first assertion is a consequence of Proposition \ref{number:non:equiv} and a counting argument.  The second and third assertions are immediate consequences of Corollary \ref{cor:key:prop}.  For the final assertion let $(V,\rho)$ be a $\G$-module.   Then the $k^\times$-module $\operatorname{Res}^{\G}_{k^\times}(V)$ admits a decomposition $$\operatorname{Res}^{\G}_{k^\times}(V) = \bigoplus_{n \in \ZZ} V_n $$ into eigenspaces.  Since the image of $k^\times$ in $\G$ is contained in the centre of $\G$ each eigenspace $V_n$ is $\G$-stable and, hence, a weight $n$ $\G$-module.
\end{proof}

\subsection{Remarks}

\subsubsection{Finite Heisenberg groups and their representations}
Let $n$ be an integer and let $r$ be the remainder obtained by dividing $n$ by $d_1$.  Let $G'$ be the subgroup
$$G' := \{(\alpha,x,y) : \text{ $\alpha \in \mu_{d_1}$, $x \in K_1$, and $y \in K_2$} \} $$ defined in the proof of Corollary \ref{cor:key:prop}.  There is a one-to-one correspondence between irreducible weight $n$ representations of $\G$ and irreducible weight $r$ representations of $G'$.  

As a consequence, for a fixed $r$, $0\leq r \leq d_1 -1$, there exists $\gcd(r,d_1)^2 \times \cdots \times \gcd(r,d_p)^2$ non-isomorphic irreducible $G'$-modules each of which has dimension 
$\frac{d_1 \times \cdots \times d_p}{\gcd(r,d_1) \times \cdots \times \gcd(r,d_p)}$.  Since
$$|G'| = d_1(d_1 \times \cdots \times d_p)^2 = \sum_{r = 0}^{d_1 - 1} \gcd(r,d_1)^2  \times \cdots \times \gcd(r,d_p)^2 \left( \frac{d_1 \times \cdots \times d_p}{\gcd(r,d_1) \times \cdots \times \gcd(r,d_p)} \right) ^2 $$ we conclude that $G'$ has exactly $\sum_{r = 0}^{d_1 - 1} \gcd(r,d_1)^2 \times \cdots \times \gcd(r,d_p)^2$ non-isomorphic irreducible $G'$-modules and, also, exactly this number of conjugacy classes.  See for example \cite[\S 2.4]{Serre:Reps}.

%\noindent {\bf Remark.}
\subsubsection{Weight $1$ representations}
If $n = 1$, then $\pi_n$ is an isomorphism and we may take the set-theoretic section $\sigma$ to be the identity map.  Also, when $n = 1$, every $\chi  \in \Hom_\ZZ(\ker \pi_n, k^\times)$ is trivial.  The resulting representation $(W_{y,\chi},\rho^{\sigma,1}_{y,\chi})$ takes the form $W_{y,\chi} := \Span_{k} \{\e_z \}_{z \in K_2}$ and an element $(\alpha,x,w)$ acts by  
$(\alpha,x,w) \cdot \e_{z} := \alpha \langle x,z\rangle \e_{z+w}$.

%:induced:rep
\subsubsection{Induced representations}\label{induced:reps}
We now show that every irreducible representation of $\G$ is induced by a $1$-dimensional representation of a suitable subgroup.  In light of Theorem \ref{mainTheorem1}, and its proof, it suffices to prove that every 
$(W_{y,\chi},\rho^{\sigma,n}_{y,\chi})$, where $y \in K_2$, $\chi \in \Hom_{\ZZ}(\ker \pi_n, k^\times)$, and $\sigma$ is a set-theoretic section of $\pi_n$, is induced by such a representation.

To achieve this fix $y \in K_2$, $\chi \in \Hom_{\ZZ}(\ker \pi_n, k^\times)$, and  define 
$V_{y,\chi} := \Span_{k} \{ \e_{y,\chi}\} \text{.}$   Let $\G(\ker \pi_n)$ denote the subgroup of $\G$ defined by 
$\G(\ker \pi_n) := \{(\alpha,x,w) : w \in \ker \pi_n \}\text{.}$  
We regard $V_{y,\chi}$ as a $\G(\ker \pi_n)$-module, by defining 
$$(\alpha,x,w)\cdot \e_{y,\chi} := \alpha^n \langle x,y \rangle \chi(w) \e_{y,\chi}$$
 for $(\alpha,x,w) \in \G(\ker \pi_n)$, and observe that the inclusion 
 $ V_{y,\chi} \rightarrow (W_{y,\chi},\rho^{\sigma,n}_{y,\chi})$,
 defined by $\e_{y,\chi} \mapsto \e_{y,\chi}$, is a  $\G(\ker \pi_n)$-homomorphism.

%\begin{proposition}
Using these considerations we check that $(W_{y,\chi},\rho^{\sigma,n}_{y,\chi})$ is isomorphic to $\Ind^{\G}_{\G(\ker \pi_n)}(V_{y,\chi})$.

%: proof of theta:reps
\subsection{Proof of Theorem \ref{thm:theta:vb:1} and Corollary \ref{unique:theta:irred}}\label{proof:theta:rep:vb}
Combining everything we are able to complete the proof of the results stated in \S \ref{main:results}.

\begin{proof}[Proof of Theorem \ref{thm:theta:vb:1}]
If $E$ is a simple semi-homogeneous vector bundle of separable type then $\G(E)$ is a non-degenerate theta group.  As a consequence, Theorem \ref{thm:theta:vb:1} is a special case of Theorem \ref{mainTheorem1}.
\end{proof}

\begin{proof}[Proof of Corollary \ref{unique:theta:irred}]
We know that $\G(E)$ has a unique irreducible weight $1$ representation.  This representation has dimension equal to $ \sqrt{\#\K(E)}$.  On the other hand $\H^{\ii(E)}(X, E)$ is a weight $1$ representation of dimension $| \chi(E) | =  \sqrt{\#\K(E)}$.  
\end{proof}

\section{Adelic theta groups and line bundles}\label{adelic:theta}

In this section we explain how to construct adelic theta groups associated to (total spaces of) line bundles on abelian varieties.  We also determine some properties of these groups and introduce some notation which we find helpful for proving Theorem \ref{adelic:NS}.  The proof of this theorem is the subject of \S \ref{proof:adelic:theta}.

\subsection{Preliminaries}
Let $X$ be an abelian variety and let $L$ be the total space of a line bundle on $X$.  If $x \in X$ then let $\Aut_x(L)$ denote the set of automorphisms of $L$ which cover $T_x$.  In what follows we let $K(L) := \{x \in X(k) : \Aut_x(L) \not = \emptyset \}$ and let $G(L)$ denote the group consisting of pairs $(x,\phi)$ where $x \in K(L)$ and where $\phi$ is an automorphism of $L$ covering $T_x$.

We now recall some of the notation introduced in \S \ref{main:results2}.  
Recall that $I$ denotes the set of positive integers which are not divisible by the characteristic of $k$, 
$$\tor(X) := \{ x \in X(k) : n x = 0 \text{ for some $n \in I$} \} \text{,}$$  and $\V(X) := \idlim \tor(X)$, where the limit is indexed by $I$ and where the maps are given by 
$[ n / m ] : \tor(X) \rightarrow \tor(X)$ whenever $m$ divides $n$.  

We identify $\V(X)$ with the set 
$$ \{ \x = (x_i)_{i \in I} : x_i \in \tor(X) \text{ and } [n / m] x_n = x_m \text{ if $n,m \in I$ and $m$ divides  $n$} \}$$ 
and let $\TT(X) := \{ \x \in \V(X) : x_1 = 0\}$.  This is a subgroup of $\V(X)$.

We now introduce some additional notation which we find helpful.
Let $L$ be the total space of a line bundle on $X$.  If $\x \in \V(X)$, then let 
$$ \supp^L(\x) := \{ n \in I : \Aut_{x_n}(n_X^* L) \not = \emptyset \}\text{.}$$

A homomorphism of abelian varieties $f  : Y \rightarrow X$ induces a homomorphism
$$ \V(f) : \V(Y) \rightarrow \V(X) \text{, defined by $\y = (y_i)_{i \in I} \mapsto (f(y_i))_{i \in I}$.} $$  We denote $\V(f)(\y) \in \V(X)$ simply by $f(\y)$ in what follows.

The following proposition plays a role in \S \ref{adelic}.  The first part can be seen as the analogue of Mumford's $4 > 2$ lemma \cite[p. 102]{MumII} in our setting.

%: supp:prop
\begin{proposition}\label{supp:prop}
Let $f : X \rightarrow Y$ be a homomorphism of abelian varieties.  Let $L$ be a line bundle on $Y$.  The following assertions hold
\begin{enumerate}
\item{if $\y \in \V(Y)$ and $m$ is the order of $y_1$, then $m \in \supp^L(\y)$;}
\item{if $\y \in \V(Y)$, $m \in \supp^L(\y)$ and $m \mid n$, then $n \in \supp^L(\y)$;}
\item{if $\y,\z \in \V(Y)$, then $\supp^L(\z) \cap \supp^L(\y) \not = \emptyset$; }
\item{if $\x \in \V(X)$, then $\supp^{f^*L}(\x) \cap \supp^L(f(\x)) \not = \emptyset$.}
\end{enumerate}
\end{proposition}

\begin{proof}
To prove (a) let $\y \in \V(Y)$ and let $m$ be the order of $y_1$.  Then $y_m \in Y_{m^2}$ and so $\Aut_{y_m}(m_Y^* L) \not = \emptyset$ because $Y_{m^2} \subseteq K(m_Y^*L)$.  Hence $m \in \supp^L(\y)$.

To prove (b) let $q = n / m$ and note that
$$n_Y^* L = q_Y^* m_Y^* L \cong q_Y^* T^*_{y_m} m_Y^* L  = T^*_{y_n} (q_Y^* m_Y^* L) = T^*_{y_n}n^*_Y L \text{.}$$

To prove (c) let $\y,\z \in \V(Y)$ and let $m$ be the least common multiple of the order of $y_1$ and $z_1$.  Then, by parts (a) and (b), $m \in  \supp^L(\y) \cap \supp^L(\z)$.

To prove (d) if $\x \in \V(X)$ and if $m$ is the order of $x_1$ then $m \in \supp^{f^*L}(\x)$ by part (a) applied to $X$ and $f^*L$.  On the other hand $f(x_m) \in Y_{m^2}$ so that $m$ is an element of $\supp^L(f(\x))$ as well.
\end{proof}

\subsection{Construction and first properties of adelic theta groups}\label{adelic}
Let $X$ be an abelian variety and let $L$ be the total space of a line bundle on $X$.  
We indicate how the \emph{adelic theta group of $L$}, which we denote by $\widehat{\G}(L)$, is constructed and discuss some of its first properties.  The construction we give here is similar to what is done in \cite[\S 7]{MumII}, which applies to the case of polarized 2-towers of abelian varieties, and \cite[Chapter 4]{theta3} which gives an approach for handling the case of the $n$-tower over $X$ obtained by considering the isogenies $n_X : X \rightarrow X$, for all natural numbers $n$.

\subsubsection{Preliminaries} Suppose that $\x \in \V(X)$, that and that $m$ and $n$ are natural numbers with $m \mid n$ and $m \in \supp^L(\x)$.  
Then there exists morphisms
$$a_{n,m}^{L,\x} : \Aut_{x_m}(m_X^* L) \rightarrow \Aut_{x_n}(n_X^* L)$$ of $k^\times$-torsors.  These morphisms are constructed in a manner similar to what is done by Mumford \cite[p. 102]{MumII} and have the
properties that: 
\begin{enumerate}
\item{if $\x \in \V(X)$, then $a_{n,m}^{L, \x} \circ a_{m,p}^{L, \x} = a_{n,p}^{L, \x}$ for all $n,m,p \in \supp^L(\x)$ whenever $p \mid m \mid n$;}
\item{if $\x \in \V(X)$, then $a^{L, \x}_{n,n} = \id_{\Aut_{x_n}(n_X ^* L)}$;}
\item{if $\x$ and $\y$ are elements of $\V(X)$, then  
$a^{L, \x + \y}_{n,m}(\phi \circ \psi) = a_{n,m}^{L, \x}(\phi) \circ a^{L, \y}_{n,m}(\psi)\text{,}$ for all $\phi \in \Aut_{x_m}(m_X^* L)$ and all $\psi \in \Aut_{y_m}(m_X^* L)$ whenever 
$m \in \supp^L(\x) \cap \supp^L(\y) \text{.}$}
\end{enumerate}

Let us indicate how the $a_{n,m}^{L,\x}$ are constructed.  To begin with we have isomorphisms of $k^\times$-torsors
$\Aut_{x_n}(q_X^* m_X^* L) \rightarrow \Aut_{x_n}(n_X^* L)$
defined by sending an element $\phi \in \Aut_{x_n}(q_X^* m_X^* L)$ to the element of $\Aut_{x_n}(n_X^* L)$ determined by the composition 
$$ n_X^* L = q_X^* m_X^* L \xrightarrow{\phi} q_X^* m_X^* L = n_X^* L\text{.}$$
Using this isomorphism we obtain pull-back morphisms of $k^\times$-torsors
$$ \Aut_{x_m}(m_X^*L) \xrightarrow{q_X^*}  \Aut_{x_n}(q_X^* m_X^*L) \rightarrow  \Aut_{x_n}(n_X^* L)$$ which we denote by $a_{n,m}^{L,\x}$.

\subsubsection{The group $\widehat{\G}(L)$}  Our construction of $\widehat{\G}(L)$ is similar to what is done by Mumford, see \cite[p. 103]{MumII}.  Indeed,  as in \cite[Definition 4, p. 103]{MumII}, we let $\widehat{\G}(L)$ denote the set of pairs $(\x, \{ \alpha_n  \}_{n \in \supp^L(\x)})$, where $ \x \in \V(X)$, $ \alpha_n \in \Aut_{x_n}(n_X^* L)$,  and 
$a_{n,m}^{\x, L} (\alpha_m) = \alpha_n$ whenever $m $ is an element of $\supp^L(\x)$ and $m \mid n  $.   Such a collection $\{ \alpha_n  \}_{n \in \supp^L(\x)}$ is what Mumford calls a compatible set of isomorphisms, see \cite[p. 102]{MumII}.  

Observe first that $\widehat{\G}(L)$ is nonempty.    Indeed,  if 
 $\x \in \V(X)$, then choose some $p \in \supp^L(\x)$, and let $\alpha_p$ be an element of $\Aut_{x_p}(p_X^* L)$.  Then   for $\ell \in \supp^L(\x)$, define
\begin{equation}\label{define:gamma1} \alpha_\ell := \begin{cases}
a^{L,\x}_{\ell,p}(\alpha_p) & \text{ if $p \mid \ell $} \\
a^{L,\x}_{\ell p,\ell} \ ^{-1} (a^{L,\x}_{\ell p,p}(\alpha_p)) & \text{ if $p \nmid \ell $ .}
\end{cases} \end{equation}
Then, using the properties of the morphisms $a^{L, \x}_{m,n}$, we check that $(\x, \{\alpha_n\}_{n \in \supp^L(\x)})$ is an element of $\widehat{\G}(L)$.
As in \cite[p. 103]{MumII}, the group operation is defined as follows.

If $(\x,\{\alpha_n\}_{n\in \supp^L(\x)})$ and $(\y,\{\beta_m\}_{m\in \supp^L(\y)}) \in \widehat{\G}(L)$ then let
\begin{equation}\label{adelic:group:operation}
(\x,\{\alpha_n\}_{n\in \supp^L(\x)})\cdot(\y,\{\beta_m\}_{m\in \supp^L(\y)}):=(\x+\y,\{\gamma_{\ell} \}_{\ell \in \supp^L(\x+\y)}) 
\end{equation}
where $\{\gamma_{\ell}\}_{\ell \in \supp^L(\x+\y)}$ is defined by choosing some element 
$p$ of $\supp^L(\x)\cap \supp^L(\y)$, which is nonempty and contained in $\supp^L(\x + \y)$, defining $\gamma_p:=\alpha_p \circ \beta_p$, which is an element of $\Aut_{x_p + y_p}(p_X^* L)$, and defining, for all $\ell \in \supp^L(\x + \y)$,
$$
\gamma_{\ell} := \begin{cases}
a^{L,\x+\y}_{\ell,p} (\gamma_p) & \text{ if $p \mid \ell $} \\
a^{L,\x+\y}_{\ell p,\ell} \ ^{-1} (a^{L,\x+\y}_{\ell p,p}(\gamma_p)) & \text{ if $p \nmid \ell$ .} \end{cases}
$$

The right hand side of \eqref{adelic:group:operation} is a well defined element of $\widehat{\G}(L)$, the pair $(\widehat{\G}(L),\cdot)$ is a group, is a central extension of $k^\times$ by $\V(X)$, and 
contains an isomorphic copy of $\TT(X)$.  

%:commutator
\subsubsection{The skew-symmetric bilinear form $[-,-]_{\widehat{\G}(L)}$}
Suppose that $(\x, \{\alpha_n \}_{n \in \supp^L(\x)})$ and $(\y, \{\beta_n\}_{n \in \supp^L(\y)})$ are elements of $\widehat{\G}(L)$.  If $p$ is an element of $\supp^L(\x) \cap \supp^L(\y)$ then $\alpha_p \circ \beta_p \circ \alpha^{-1}_p \circ \beta^{-1}_p$ corresponds to a unique $\gamma \in k^\times$ which is independent of our choice of $p$.  

Also
if $[-,-]_{\widehat{\G}(L)}$ denotes the commutator of $\widehat{\G}(L)$ then
$$ [(\x, \{\alpha_n\}_{n \in \supp^L(\x)}),(\y,\{\beta_n\}_{n \in \supp^L(\y)})]_{\widehat{\G}(L)} = (\mathbf{0},\{ \gamma \id_{n^*_X L}\}_{n \in \supp^L(\mathbf{0})}) \text{.}$$
As a consequence we obtain a skew-symmetric bilinear form
$$[-,-]_{\widehat{\G}(L)} : \V(X) \times \V(X) \rightarrow k^\times \text{, defined by $[\x,\y]_{\widehat{\G}(L)} = \gamma$.}$$ 

Observe also that if $\x,\y \in \V(X)$, if $p \in \supp^L(\x) \cap \supp^L(\y)$, and if 
$$[-,-]_{G(p_X^* L)} : K(p_X^* L) \times K(p_X^* L) \rightarrow k^\times$$ 
denotes the skew-symmetric bilinear form determined by the commutator of the theta group $G(p_X^* L)$,
then 
$  [\x,\y]_{\widehat{\G}(L)} = [x_p,y_p]_{G(p_X^* L)}\text{.}$

\subsubsection{The group homomorphism $\widehat{\G}(f)$}\label{pullback}
%: pullback
Similar to what is done in \cite[Proposition 4.9, p. 51]{theta3} we check that our construction of adelic theta groups behaves well with respect to isogenies.  

Let $f : X \rightarrow Y$ be a homomorphism of abelian varieties and let $L$ be a line bundle on $Y$.  We construct a group homomorphism $\widehat{\G}(f) : \widehat{\G}(f^* L) \rightarrow \widehat{\G}(L)$ which fits into the commutative diagram
\begin{equation}\label{pull:back:theta:diagram}\xymatrix{  
1 \ar[r] & k^\times \ar[r] \ar @{=} [d] & \widehat{\G}(f^* L) \ar[r] \ar[d]^-{\widehat{\G}(f)}& \V(X) \ar[r] \ar[d]^-{\V(f)} & 0 \\
1 \ar[r] & k^\times \ar[r] & \widehat{\G}( L) \ar[r] & \V(Y) \ar[r] & 0 \text{.}
} \end{equation}
In other words the extension determined by $\widehat{\G}(f^* L)$ is equivalent to the pull-back, with respect to the group homomorphism $\V(f) : \V(X) \rightarrow \V(Y)$, of the extension determined by $\widehat{\G}(L)$.

First of all if $\x \in \V(X)$ and  $m \in \supp^L(f(\x))$, then there exists isomorphisms 
$$b^{f,L}_{\x,m} : \Aut_{x_m}(m_X^* f^* L) \rightarrow  \Aut_{f(x_m)}(m_Y^* L)$$
of $k^\times$-torsors.
These morphisms have the properties that:  
\begin{enumerate}
\item{if $\x \in \V(X)$ and $m \mid n$ then the diagram
$$ 
\xymatrix{ \Aut_{x_m}(m_X^* (f^* L) ) \ar[r]^-{ a^{f^*L, \x}_{n,m} } \ar[d]_-{b^{f,L}_{\x,m} } & \Aut_{x_n}(n_X^* ( f^* L))  \ar[d]^-{b^{f,L}_{\x,n}} \\
\Aut_{f(x_m)}(m_Y^* L) \ar[r]^-{a^{L, f(\x)}_{n,m}} & \Aut_{f(x_n)}(n_Y^* L)
} $$ commutes;
 }
 \item{ if $\x, \z \in \V(X)$, then 
 $b^{f,L}_{\x + \z,m}(\phi \circ \psi) = b^{f,L}_{\x,m}(\phi) \circ b^{f,L}_{\z,m}(\psi) $
for all $m \in \supp^{f^* L}(\x) \cap \supp^{f^* L}(\z)$,  $\phi \in \Aut_{x_m}(m_X^* (f^* L))$, and  $\psi \in \Aut_{z_m}(m_X^*(f^* L))$. }
\end{enumerate}
The morphisms $b^{f,L}_{\x,m}$ can be constructed in a manner similar to the method employed in the proof of \cite[Proposition 4.9]{theta3}. 
 More explicitly first note that, since $m_Y \circ f = f \circ m_X$, we have isomorphisms 
$\Aut_{x_m}(f^* m_Y^* L) \rightarrow \Aut_{x_m}(m_X^* f^* L) \text{.}$ Composing the inverse of these isomorphisms with the inverse of the pullback morphisms 
$f^* : \Aut_{f(x_m)}(m_Y^* L) \rightarrow \Aut_{x_m}(f^* m_Y^* L)$ yields the isomorphisms
$b^{f,L}_{\x,m}$.

To construct the group homomorphism $\widehat{\G}(f) : \widehat{\G}(f^* L) \rightarrow \widehat{\G}(L)$, let $(\x ,\{\alpha_n \}_{n \in \supp^{f^*L}(\x)})$ be an element of  $\widehat{\G}(f^*L)$, let $m$ be the order of $x_1$ and let $\y := f(\x)$.  Then $m \in \supp^{f^*L}(\x) \cap \supp^L(\y)$ and we can let $\beta_m := b^{f,L}_{\x,m}(\alpha_m)$ and define, for $p \in \supp^L(\y)$, 
$$ \beta_p := \begin{cases}
a^{L,\y}_{p,m} (\beta_m) & \text{ if $m \mid p $} \\
a^{L,\y}_{pm ,p} \ ^{-1} a^{L,\y}_{p m,m} (\beta_m) & \text{ if $m \nmid p $.}
\end{cases}
$$
The image of $(\x ,\{\alpha_n \}_{n \in \supp^{f^*L}(\x)})$ under $\widehat{\G}(f)$ is now defined to be
$ (\y, \{ \beta_n \}_{n \in \supp^L(\y)})\text{.}$ 
%:pullback:homo
Using the properties of the morphisms $b^{f,L}_{\x,m}$, we check that the above definition defines a group homomorphism 
$\widehat{\G}(f) : \widehat{\G}(f^* L) \rightarrow \widehat{\G}(L)$ fitting into the commutative diagram \eqref{pull:back:theta:diagram}.

\section{The group homomorphism $\operatorname{NS}(X) \hookrightarrow \H^2(\V(X),k^\times)$}\label{proof:adelic:theta}

Before proving Theorem \ref{adelic:NS} we make some auxiliary remarks.

\subsection{The cohomology group $\H^2(\V(X),k^\times)$}  
We consider $k^\times$ as a trivial $\V(X)$-module and let $\H^2(\V(X),k^\times)$ denote the group of normalized $2$-cocycles $\V(X) \times \V(X) \rightarrow k^\times$ modulo coboundaries.  Recall that there is a 1-1 correspondence between central extensions of $k^\times$ by $\V(X)$ and elements of $\H^2(\V(X),k^\times)$.  In addition, since $k^\times$ is divisible, $\operatorname{Ext}^1_{\ZZ}(\V(X),k^\times) = 0$ so every abelian central extension of $k^\times$ by $\V(X)$ is trivial.  

\subsection{The map $\widehat{\G} : \Pic(X) \rightarrow \H^2(\V(X),k^\times)$}\label{Pic:adelic:map}
If $L$ and $M$ are isomorphic line bundles then it is clear that their adelic theta groups $\widehat{\G}(L)$ and $\widehat{\G}(M)$ are equivalent extensions of $k^\times$ by $\V(X)$.  We thus have a well-defined map
$\widehat{\G} : \Pic(X) \rightarrow \H^2(\V(X),k^\times) $
defined by sending the isomorphism class of a line bundle $L$ to the equivalence class of the extension determined by its adelic theta group $\widehat{\G}(L)$.  We let $[\widehat{\G}(L)]$ denote the equivalence class of the extension of $k^\times$ by $\V(X)$ which is determined by $\widehat{\G}(L)$.

\subsection{Outline of proof} To prove Theorem \ref{adelic:NS} we first prove that the map $$\widehat{\G} : \Pic(X) \rightarrow   \H^2(\V(X),k^\times) \text{,}$$ constructed in \S \ref{Pic:adelic:map}, is a group homomorphism.   To do this we determine a relationship amongst the adelic theta groups $\widehat{\G}(L)$, $\widehat{\G}(M)$, and $\widehat{\G}(L \otimes M)$.   Indeed we prove that the extension class $[\widehat{\G}(L \otimes M)]$ is the Baer sum $[\widehat{\G}(L)] + [\widehat{\G}(M)]$ of the extensions classes $[\widehat{\G}(L)]$ and $[\widehat{\G}(M)]$.  Note that this behaviour is in contrast to the behaviour of the theta groups $G(L)$, $G(M)$, and $G(L\otimes M)$; the main issue being that their is no clear relationship, in general, amongst the groups $K(L)$, $K(M)$, and $K(L\otimes M)$.

Our next task is to determine the kernel of $\widehat{\G}$.  
To this end let $L$ be a line bundle on $X$.   We first relate the commutator $[-,-]_{\widehat{\G}(L)}$ of $\widehat{\G}(L)$ to the (non-degenerate) Weil-pairing 
$ \overline{e_n} : X_n \times \widehat{X}_n \rightarrow \mu_n\text{,}$
defined for all integers $n$ which are relatively prime to the characteristic of $k$. (See for instance \S 20, p. 170 of \cite{Mum}.) We then use this relationship to prove that $\widehat{\G}(L)$ is abelian if and only if $L \in \Pic^0(X)$.  This fact, combined with the fact that $\operatorname{Ext}^1_\ZZ(\V(X),k^\times) = 0$, implies that $\Pic^0(X)$ is the kernel of $\widehat{\G}$. We thus obtain a well-defined injective group homomorphism $\widehat{\G} : \NS(X) \hookrightarrow \H^2(\V(X),k^\times)$, defined by sending the class of a line bundle $L \in \NS(X)$ to the class of the extension determined by its adelic theta group $\widehat{\G}(L)$.

Finally, to complete the proof of Theorem \ref{adelic:NS}, we use \S \ref{pullback} to check that the group homomorphism $\widehat{\G} : \NS(X) \hookrightarrow \H^2(\V(X),k^\times)$ behaves well under isogenies.

\begin{proof}[Proof of Theorem \ref{adelic:NS}]

\noindent
Step 1.  %: key:adelic:relation
Let $L$ and $M$ be line bundles on $X$.  The relation $$ [\widehat{\G}(L \otimes M)]=[\widehat{\G}(L)]+[\widehat{\G}(M)] $$ holds in  
$\H^2(\VV(X), k^\times)\text{.}$

To prove Step 1 we define normalized set-theoretic sections 
$$\sigma_L : \VV(X)\rightarrow \widehat{\G}(L) \text{, } \sigma_M : \VV(X)\rightarrow \widehat{\G}(M)\text{, and }
\sigma_{L\otimes M} : \VV(X)\rightarrow \widehat{\G}(L \otimes M)$$ and check that the relation  
$ [-,-]_{\sigma_{L\otimes M}} = [-,-]_{\sigma_L} + [-,-]_{\sigma_M} $ holds amongst the corresponding factor sets. To establish this we must show that if $\x$ and $\y$ are elements of $\VV(X)$, if   
$$ \sigma_{L}(\x)\cdot \sigma_{L}(\y) \cdot \sigma_{L}(\x+\y)^{-1}=(\mathbf{0}, \{\alpha \id_{n_X^*L}\}_{n\in \supp^{L}(\mathbf{0})}) \text{, } \alpha \in k^\times \text{,}$$
if 
$$\sigma_{M}(\x)\cdot \sigma_{M}(\y) \cdot \sigma_{M}(\x+\y)^{-1}= (\mathbf{0}, \{\beta \id_{n_X^*M}\}_{n\in \supp^{M}(\mathbf{0})}) \text{, } \beta \in k^\times \text{,}$$ and if 
$$ \sigma_{L\otimes M}(\x)\cdot \sigma_{L\otimes M}(\y) \cdot \sigma_{L\otimes M}(\x+\y)^{-1}=(\mathbf{0}, \{\gamma \id_{n_X^*(L\otimes M)}\}_{n\in \supp^{L\otimes M}(\mathbf{0})}) \text{, } \gamma \in k^\times 
$$
then $\gamma =\alpha \beta$.  Observe that this holds if $\gamma \id_{n_X^*(L\otimes M)}=\alpha \id_{n_X^*L}\otimes \beta \id_{n_X^*M}$ for some $n$.

We now define sections $\sigma_L$, $\sigma_M$ and $\sigma_{L\otimes M}$ with the desired properties. First of all for any element $\x$ of $\VV(X)$ let $m_{\x}$ be the order of $x_1$.  

Now let $\x \in \V(X)$.  If $\x = \mathbf{0}$ then define $\alpha_{m_{\x}}^{\x}:=\id_L$ and $\beta_{m_{\x}}^{\x}:=\id_M$.  Otherwise 
 choose $\alpha_{m_{\x}}^{\x} \in \Aut_{x_{m_{\x}}}(m_{\x}^* \ _X L)$, choose $\beta_{m_{\x}}^{\x} \in \Aut_{x_{m_{\x}}}(m_{\x}^* \ _X M)$ and  set
$$\gamma_{m_{\x}}^{\x}:= \alpha_{m_{\x}}^{\x}\otimes \beta_{m_{\x}}^{\x} \in \Aut_{x_{m_{\x}}}(m_{\x}^* \ _X(L \otimes M)) \text{.}$$  
Then, 
$\alpha_{m_{\x}}^{\x}$, $\beta_{m_{\x}}^{\x}$ and $\gamma_{m_{\x}}^{\x}$ determine (unique) elements of $\widehat{\G}(L)$, $\widehat{\G}(M)$, and $\widehat{\G}(L \otimes M)$ and hence allow us to define normalized sections $\sigma_L$, $\sigma_M$ and $\sigma_{L\otimes M}$. 

Now let $\x$ and $\y$ be elements of $\VV(X)$ and let $p:=\lcm(m_{\x}, m_{\y})$.  
Let 
$$ \text{$\phi_p \in \Aut_{0}(p_X^*L)$, $\psi_p \in  \Aut_{0}(p_X^*M)$ and $\eta_p \in \Aut_{0}(p_X^*(L\otimes M))$}$$
be such that 
$$ (0,\phi_p)=(x_p,\alpha^{\x}_p)\cdot (y_p,\alpha^{\y}_p)\cdot (x_p+y_p,\alpha_p^{\x+\y})^{-1} \in G(p_X^*L) \text{,}$$
$$(0,\psi_p)=(x_p,\beta^{\x}_{p})\cdot (y_p,\beta^{\y}_p)\cdot (x_p+y_p,\beta_p^{\x+\y})^{-1} \in G(p_X^*M) \text{, and} $$
$$(0,\eta_p)=(x_p,\gamma^{\x}_{p})\cdot (y_p,\gamma^{\y}_p)\cdot (x_p+y_p,\gamma_p^{\x+\y})^{-1} \in G(p_X^*(L\otimes M))\text{.}$$
(These are the $p$th components of $[\x,\y]_{\sigma_L}$, $[\x,\y]_{\sigma_M}$ and $[\x,\y]_{\sigma_{L\otimes M}}$.)  Since
$$\gamma_p^{\x}=\alpha_p^{\x}\otimes \beta_p^{\x}, \gamma_p^{\y}=\alpha_p^{\y}\otimes \beta_p^{\y} \textrm{ and } \gamma_p^{\x+\y}=\alpha_p^{\x+\y}\otimes \beta_p^{\x+\y}\text{,}$$ 
computing the above multiplications
we conclude that $\eta_p=\phi_p\otimes \psi_p$.

\noindent
Step 2:  If $L$ is a line bundle on $X$ then $\widehat{\G}(L)$ is abelian if and only if $L \in \Pic^0(X)$.

%: Weil-pairing
Assume that $\widehat{\G}(L)$ is abelian.  To prove that $L$ is an element of $\Pic^0(L)$ we relate $[-,-]_{\widehat{\G}(L)}$ to the Weil-pairing.  To this end,  let $n$ be a positive integer not divisible by the characteristic of $k$.  
If $x \in X_n$, $y \in n_X^{-1}(K(L))$, and $z \in X$ is such that $nz = y$, then 
$ \overline{e_n}(x,\phi_L(y)) = [x,z]_{G(n^*L)}\text{,}$ by \cite[p. 212]{Mum}.

We now prove that if $[\x,\y]_{\widehat{\G}(L)} = 1$, for all $\x, \y \in \V(X)$, then $\phi_L(y) = \Osh_X$ for all $y \in X$.
To accomplish this, we first make the following reduction.  
Since $\tor(X)$ is Zariski dense in $X$, to prove that $[\x,\y]_{\widehat{\G}(L)} = 1$, for all $\x, \y \in \V(X)$, implies that $\phi_L(y) = \Osh_X$ for all $y \in X$, it suffices to show that if $[\x,\y]_{\widehat{\G}(L)} = 1$, for all $\x, \y \in \V(X)$ then  $\phi_L(y) = \Osh_X$ for all $y \in \tor(X)$.  

To establish this reduction step, let $n$ be the order of $y$.  Then $y \in K(n_X^* L)$, and so $y \in n_X^{-1}(K(L))$.  Now choose $\z \in \V(X)$ such that $z_1 = y$.  Then $n z_n = y$ and also $n^2 z_n = n y = 0$ so $z_n \in K(n_X^*L)$.  Hence $n \in \supp^L(\z)$.

Let $x \in X_n$ and choose $\x \in \TT(X)$ with $x_n = x$.
We then have
\begin{equation}\label{Weil:eqn}
\overline{e_n}(x,\phi_L(y)) = [x,z_n]_{G(n^* L)} = [\x, \z]_{\widehat{\G}(L)} =1 \text{,}
\end{equation}
where the second rightmost equality follows because $n \in \supp^L(\x) \cap \supp^L(\z)$, and where the rightmost equality follows because $\widehat{\G}(L)$ is abelian.  

Since $x$ is an arbitrary element of $X_n$ the relation \eqref{Weil:eqn} holds for all $x \in X_n$.  Since $\overline{e_n}$ is non-degenerate this means that $\phi_L(y) = \Osh_X$ which is what we wanted to show.

The above implies that if $\widehat{\G}(L)$ is abelian then $L \in \Pic^0(X)$.
Indeed if $\widehat{\G}(L)$ is abelian then $[\x,\y]_{\widehat{\G}(L)} = 1$ for all $\x,\y \in \V(X)$.  This implies that $\phi_L(y) = \Osh_X$ for all $y \in X$.  Hence $L \in \Pic^0(X)$.

Conversely if $L \in \Pic^0(X)$ then $G(L)$ is abelian which implies that $\widehat{\G}(L)$ is abelian.

\noindent
Step 3.  The homomorphism $\widehat{\G} : \operatorname{NS}(X) \hookrightarrow \H^2(\V(X), k^\times)$ is functorial in $X$.

Let $f : X \rightarrow Y$ be a homomorphism of abelian varieties.  Using the definition of $\widehat{\G}(f)$ (see \S \ref{pullback}) we check that the diagram (which has exact rows)
$$\xymatrix{ 1 \ar[r] & \Pic^0(X) \ar[r] & \Pic(X) \ar[r]^-{\widehat{\G}} & \H^2(\V(X),k^\times) \\
1  \ar[r] & \Pic^0(Y) \ar[r] \ar[u]^-{f^*}& \Pic(Y) \ar[r]^-{\widehat{\G}} \ar[u]^-{f^*} & \ar[u]^-{f^*} \H^2(\V(Y), k^\times)
 }$$ commutes.
\end{proof}

\providecommand{\bysame}{\leavevmode\hbox to3em{\hrulefill}\thinspace}
\providecommand{\MR}{\relax\ifhmode\unskip\space\fi MR }
\providecommand{\MRhref}[2]{%
  \href{http://www.ams.org/mathscinet-getitem?mr=#1}{#2}
}
\providecommand{\href}[2]{#2}


\begin{thebibliography}{10}

\bibitem{BL}
C.~Birkenhake and H.~Lange, \emph{Complex abelian varieties}, second ed.,
  Springer-Verlag, Berlin, 2004.

\bibitem{B-L-R}
S.~Bosch, W.~Lutkebohmert, and M.~Raynaud, \emph{Neron models},
  Springer-Verlag, Berlin, 1990.

\bibitem{Bri}
M.~Brion, \emph{Homogeneous projective bundles over abelian varieties}, Algebra
  and Number Theory \textbf{7} (2013), no.~10, 2475--2510.

\bibitem{Goren:theta:preprint}
E.~Z. Goren, \emph{Quasi-symmetric line bundles on abelian varieties}, Max
  Planck Inst, Preprint (1995), 67.

\bibitem{Grieve:PhD:Thesis}
N.~Grieve, \emph{Topics related to vector bundles on abelian varieties}, Ph.D.
  thesis, Queen's University, 2013.

\bibitem{Grieve-cup-prod-ab-var}
\bysame, \emph{Index conditions and cup-product maps on {A}belian varieties},
  Internat. J. Math. \textbf{25} (2014), no.~4, 1450036, 31.

\bibitem{G-P}
M.~Gross and S.~Popescu, \emph{Equations of $(1,d)$-polarized abelian
  surfaces}, Math. Ann. \textbf{310} (1998), no.~2, 333--377.

\bibitem{MacLane:homology}
S.~ MacLane, \emph{Homology}, Springer-Verlag, Berlin, 1995.

\bibitem{Muk78}
S.~Mukai, \emph{Semi-homogeneous vector bundles on an {A}belian variety}, J.
  Math. Kyoto Univ. \textbf{18} (1978), no.~2, 239--272.

\bibitem{MumI}
D.~Mumford, \emph{On the equations defining abelian varieties. {I}}, Invent.
  Math. \textbf{1} (1966), 287--354.

\bibitem{MumII}
\bysame, \emph{On the equations defining abelian varieties. {II}}, Invent.
  Math. \textbf{3} (1967), 75--135.

\bibitem{MumIII}
\bysame, \emph{On the equations defining abelian varieties. {III}}, Invent.
  Math. \textbf{3} (1967), 215--244.

\bibitem{Mum:Quad:Eqns}
\bysame, \emph{Varieties defined by quadratic equations}, Questions on
  {A}lgebraic {V}arieties ({C}.{I}.{M}.{E}., {III} {C}iclo, {V}arenna, 1969),
  Edizioni Cremonese, Rome, 1970, pp.~29--100.

\bibitem{theta3}
\bysame, \emph{Tata lectures on theta. {III}}, Birkh\"auser Boston Inc., 2007.

\bibitem{Mum}
\bysame, \emph{Abelian varieties}, Published for the Tata Institute of
  Fundamental Research, Bombay, 2008, Corrected reprint of the second (1974)
  edition.

\bibitem{Nakamura:99}
I.~Nakamura, \emph{Stability of degenerate abelian varieties}, Invent. Math.
  \textbf{136} (1999), 659--715.

\bibitem{Oda:deRham}
T.~Oda, \emph{The first de {R}ham cohomology group and {D}ieudonne modules},
  Ann. Sc. Ecole Norm. Sup. Paris \textbf{4} (1969), no.~2, 63--135.

\bibitem{OdaElliptic71}
\bysame, \emph{Vector bundles on an elliptic curve}, Nagoya Math. J.
  \textbf{43} (1971), 41--72.

\bibitem{Oprea:2011}
D.~Oprea, \emph{The {V}erlinde bundles and the semihomogeneous {W}irtinger
  duality}, J. reine angew. Math. \textbf{654} (2011), 181--217.

\bibitem{Serre:Reps}
J.~P. Serre, \emph{Linear representations of finite groups}, Springer-Verlag,
  New York, 1977.

\bibitem{Shin}
S.~W. Shin, \emph{Abelian varieties and {W}eil representations}, Algebra and
  Number Theory \textbf{6} (2012), no.~8, 1719--1772.

\bibitem{Umemura:1973}
H.~Umemura, \emph{Some results in the theory of vector bundles}, Nagoya Math.
  J. \textbf{52} (1973), 97--128.

\bibitem{Umemura:1976}
\bysame, \emph{On a certain type of vector bundles over an {A}belian variety},
  Nagoya Math. J. \textbf{64} (1976), 31--45.

\end{thebibliography}
\end{document}